\documentclass[oneside,english]{amsart}
\usepackage[T1]{fontenc}
\usepackage[latin9]{inputenc}
\usepackage{geometry}
\usepackage{amsthm}
\usepackage{hyperref}
\usepackage{tikz}
\usepackage{kbordermatrix}
\makeatletter
\numberwithin{equation}{section}
\numberwithin{figure}{section}
\theoremstyle{plain}
\newtheorem{thm}{\protect\theoremname}
\theoremstyle{plain}

\theoremstyle{plain}

\theoremstyle{plain}

\theoremstyle{plain}

\theoremstyle{plain}

\theoremstyle{plain}

\makeatother

\usepackage{babel}
\usepackage{colortbl}

\providecommand{\conjecturename}{Conjecture}
\providecommand{\lemmaname}{Lemma}
\providecommand{\theoremname}{Theorem}
\providecommand{\problemname}{Problem}
\providecommand{\propositionname}{Proposition}
\providecommand{\corollaryname}{Corollary}
\providecommand{\definitionname}{Definition}
\begin{document}

\title{The Proportion of Trees that are Linear}

\author{Tanay Wakhare$^\ast$, Eric Wityk$^\dag$, and Charles R. Johnson$^\S$}

\thanks{{\scriptsize
\hskip -0.4 true cm MSC(2010): Primary: 05C30; Secondary: 05C05.
$^\S$ Corresponding author
}}


\address{$^\ast$~University of Maryland, College Park, MD 20742, USA}
\email{twakhare@gmail.com}
\address{$^\dag$~Georgia Institute of Technology, Atlanta, GA 30332, USA}
\email{eric.wityk@gmail.edu}
\address{$^\S$~College of William and Mary, Williamsburg, VA 23185, USA}
\email{crjohn@wm.edu}

\maketitle

\begin{abstract}
We study several enumeration problems connected to \textit{linear trees}, a broad class which includes stars, paths, generalized stars, and caterpillars. We provide generating functions for counting the number of linear trees on $n$ vertices, characterize the asymptotic growth rate of the number of nonisomorphic linear trees, and show that the distribution of $k$-linear trees on $n$ vertices follows a central limit theorem.
\end{abstract}

\section{Introduction}
A \textit{high degree vertex} (HDV) in a simple undirected graph is one of degree at least $3$. A tree is called \textit{linear} if all of its HDV's lie on a single induced path, and \textit{$k$-linear} if there are $k$ HDV's. The linear trees include the familiar classes of paths, stars, generalized stars (g-stars, with exactly one HDV), double g-stars \cite{Johnson3}, and caterpillars \cite{Harary}, etc. They have become important, as all multiplicity lists of eigenvalues occurring among Hermitian matrices, whose graph is a given linear tree, may be constructed via a \textit{linear superposition principal} (LSP) that respects the precise structure of the linear tree \cite{Johnson3, TwoAndrews}. For other, \textit{nonlinear} trees, multiplicity lists require different methodology. For a tree to be nonlinear, there must be at least $4$ HDV's (and at least $10$ vertices altogether). An example of a nonlinear tree and a linear tree, both on $13$ vertices, is given in Figure \ref{linnonlin}.



\begin{center}
\begin{figure}[!ht]
\centering
\def\r{1.9}
\def\k{6pt}
\begin{tikzpicture}[scale=0.7,
rel/.style={circle, draw, only marks, mark=*, mark options={fill=red},mark size=\k},
bor/.style={rectangle,draw,rounded corners=0.6ex, minimum size=4pt, inner sep=5pt},
coo/.style={circle, draw, only marks, mark=*, fill=red, minimum size=\k+3pt},
]
\foreach \x/\xtext in {1,2,3}{
\draw (0pt,0pt) -- ({\r*cos(\x*2*pi/3 r)}, {\r*sin(\x*2*pi/3 r)});}
\draw plot[only marks, mark=*, mark options={fill=red},mark size=\k] 
coordinates{(0,0)} node[] {};
\foreach \y/\xtext in {1,2,3}{
\draw[shift={({\r*cos(2*pi/3 r)}, {\r*sin(2*pi/3 r)})}]
(0pt,0pt) -- ++({\r*cos((2*pi/3+pi/4*(\y-2)) r)}, {\r*sin((2*pi/3+pi/4*(\y-2)) r)}) 
plot[only marks, mark=*, mark options={fill=white},mark size=\k] 
coordinates{({\r*cos((2*pi/3+pi/4*(\y-2)) r)}, {\r*sin((2*pi/3+pi/4*(\y-2)) r)})};}
\foreach \y/\xtext in {1,2,3}{
\draw[shift={({\r*cos(4*pi/3 r)}, {\r*sin(4*pi/3 r)})}] (0pt,0pt) 
-- ++({\r*cos((4*pi/3+pi/4*(\y-2)) r)}, {\r*sin((4*pi/3+pi/4*(\y-2)) r)})
plot[only marks, mark=*, mark options={fill=white},mark size=\k] 
coordinates{({\r*cos((4*pi/3+pi/4*(\y-2)) r)}, {\r*sin((4*pi/3+pi/4*(\y-2)) r)})};}
\foreach \y/\xtext in {1,2,3}{
\draw[shift={({\r*cos(0 r)}, {\r*sin(0 r)})}] (0pt,0pt) 
-- ++({\r*cos((pi/4*(\y-2)) r)}, {\r*sin((pi/4*(\y-2)) r)})
plot[only marks, mark=*, mark options={fill=white},mark size=\k] 
coordinates{({\r*cos((pi/4*(\y-2)) r)}, {\r*sin((pi/4*(\y-2)) r)})};}
\foreach \x/\xtext in {1,2,3}{
\draw plot[only marks, mark=*, mark options={fill=red},mark size=\k] 
coordinates{({\r*cos(\x*2*pi/3 r)}, {\r*sin(\x*2*pi/3 r)})};}
\end{tikzpicture}
\hspace{1cm}
\begin{tikzpicture}[scale=0.7,
bor/.style={circle, draw, only marks, mark=*, fill=white,mark size=\k},
rel/.style={rectangle,draw,rounded corners=0.6ex,  fill=white, minimum size=\k},
coo/.style={circle, draw, only marks, mark=*, fill=red, minimum size=\k+3pt},
]
\draw (-2*\r,0) node[bor]{}-- (0,0) node[coo]{}-- (0,-\r)node[bor]{}  --(0,\r)node[bor]{} --(0,2*\r)node[bor]{} --(0,0) -- (\r,0)node[bor]{} --(3*\r,0)node[bor]{};
\foreach \x/\xtext in {3,5}{
\draw[shift={(-\r,0)}] (0pt,0pt) node[coo]{} -- ({\r*cos(\x*pi/4 r)}, {\r*sin(\x*pi/4 r)})
node[bor]{};}
\foreach \x/\xtext in {-1,1}{
\draw[shift={(2*\r,0)}] (0pt,0pt) node[coo]{}-- ({\r*cos(\x*pi/4 r)}, {\r*sin(\x*pi/4 r)})
node[bor]{};
}
\end{tikzpicture}

\caption{Nonlinear and $3$-linear trees on 13 vertices (HDVs in red)}
\label{linnonlin}
\end{figure}
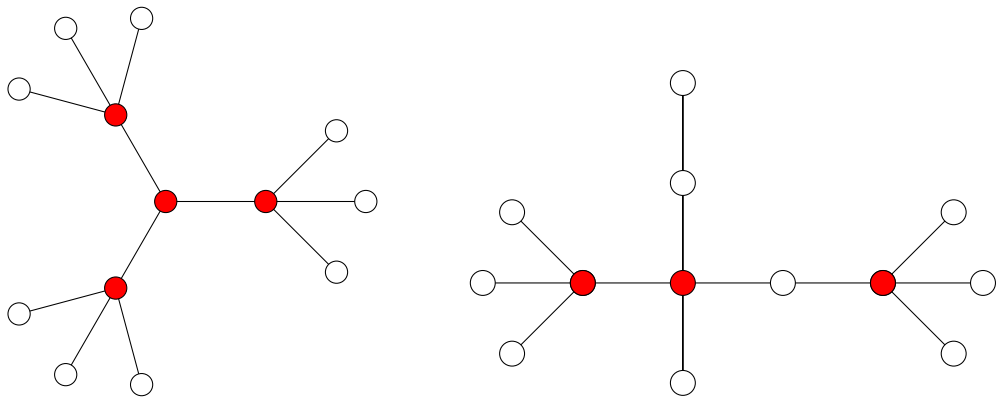
\end{center}

Linear trees are a substantial generalization of caterpillars, and the problem of counting the number of non-isomorphic linear trees is significantly harder than for caterpillars. We define a bivariate generating function for the number of $k$-linear trees on $n$ vertices, which enables the fast computation of these numbers.  Additionally, we are able to obtain asymptotic growth rates which show that the probability that a randomly chosen tree on $n$ vertices will be linear approaches $0$ as $n\to \infty$. This shows that while the LSP is a useful characterization, it has limited applicability to studying the spectra of general trees. As $n$ increases, the LSP characterizes the spectra of an asymptotically vanishing proportion of all trees. However, the proportion of linear trees vanishes slowly, so that the LSP is very important, especially for small numbers of vertices. We conclude with an investigation of the distribution of $k$-linear trees on $n$ vertices, and show that this satisfies a central limit theorem.



\section{Generating Functions}

There are strong links between nonisomorphic linear trees and partitions, which are famously difficult to enumerate.  In constructing a generating function for $k$-linear trees on $n$ vertices, we will rely the generating function for integer partitions. Let $$P(x) := \prod_{i=1}^\infty \frac{1}{1-x^i} = \sum_{n=0}^\infty p(n)x^n$$ denote the generating function for $p(n)$, the number of unrestricted partitions of $n$. Let $r_{n,k}$ be the number of reflections of linear trees on $n$ vertices with $k$ HDV's (which counts linearly symmetric trees once and linearly asymmetric trees twice), and let $s_{n,k}$ denote the number of linearly symmetric trees on $n$ vertices with $k$ HDV's. Letting $a_{n,k}$ denote the number of non-isomorphic $k$-linear trees on $n$ vertices, we deduce $$a_{n,k}=\frac12 \left(r_{n,k}+s_{n,k}\right). $$
The following generating function allows us to compute recurrences for the coefficients which allow for fast computation of $a_{n,k}$.

{\scriptsize
\begin{table}\caption{\cite[Appendix A]{Wityk} The number of $k$-linear trees on $n$ vertices}\label{table1}
\begin{tabular}{|r|r|r|r|r|r|r|r|r|r|r|r|r|} \hline
$n$ & $k=1$ & $k=2$  & $k=3$     & $k=4$     & $k=5$      & $k=6$      & $k=7$     & $k=8$   & $k=9$  & $k=10$ & $k=11$ & Total      \\ \hline\hline
10  & 25    & 56     & 22        & 1         & 0          & 0          & 0         & 0       & 0      & 0      & 0      & 105        \\
\rowcolor{gray!25}
11  & 36    & 114    & 74        & 6         & 0          & 0          & 0         & 0       & 0      & 0      & 0      & 231        \\
12  & 50    & 224    & 219       & 37        & 1          & 0          & 0         & 0       & 0      & 0      & 0      & 532        \\
\rowcolor{gray!25}
13  & 70    & 441    & 576       & 158       & 8          & 0          & 0         & 0       & 0      & 0      & 0      & 1,254      \\
14  & 94    & 733    & 1,394     & 591       & 58         & 1          & 0         & 0       & 0      & 0      & 0      & 2,872      \\
\rowcolor{gray!25}
15  & 127   & 1,252  & 3,150     & 1,896     & 304        & 9          & 0         & 0       & 0      & 0      & 0      & 6,739      \\
16  & 168   & 2,091  & 6,733     & 5,537     & 1,342      & 82         & 1         & 0       & 0      & 0      & 0      & 15,955     \\
\rowcolor{gray!25}
17  & 222   & 3,393  & 13,744    & 14,812    & 5,085      & 508        & 11        & 0       & 0      & 0      & 0      & 37,776     \\
18  & 288   & 5,408  & 26,969    & 37,133    & 17,232     & 2,635      & 112       & 1       & 0      & 0      & 0      & 89,779     \\
\rowcolor{gray!25}
19  & 375   & 8,440  & 51,185    & 87,841    & 53,200     & 11,523     & 804       & 12      & 0      & 0      & 0      & 213,381    \\
20  & 480   & 12,982 & 94,323    & 198,267   & 152,316    & 44,704     & 4,730     & 145     & 1      & 0      & 0      & 507,949    \\
\rowcolor{gray!25}
21  & 616   & 19,650 & 169,453   & 429,199   & 409,105    & 156,513    & 23,451    & 1,182   & 14     & 0      & 0      & 1,209,184  \\
22  & 781   & 29,388 & 297,533   & 896,731   & 1,040,846  & 504,869    & 102,186   & 7,862   & 184    & 1      & 0      & 2,880,382  \\
\rowcolor{gray!25}
23  & 990   & 43,394 & 512,006   & 1,814,978 & 2,526,691  & 1,517,918  & 400,074   & 43,602  & 1,682  & 15     & 0      & 6,861,351  \\
24  & 1,243 & 63,430 & 865,050   & 3,572,810 & 5,887,488  & 4,300,385  & 1,434,484 & 211,388 & 12,381 & 226    & 1      & 16,348,887 \\
\rowcolor{gray!25}
25  & 1,562 & 91,754 & 1,437,739 & 6,858,774 & 13,231,478 & 11,567,238 & 4,773,006 & 915,546 & 75,951 & 2,288  & 17     & 38,955,354 \\ \hline
\end{tabular}
\end{table}}  

\begin{thm}\label{thm1}
The generating function for $k$-linear trees on $n$ vertices is

\begin{align*}
2\sum_{n=1}^\infty \sum_{k=0}^\infty a_{n,k}x^ny^k &=  \left(P(x) - \frac{1}{1-x}\right)^2 \frac{x^2y^2}{\left( 1 -x+xy -xyP(x)  \right)} \\
&+\frac{1}{1-x} \left(P(x^2) - \frac{1}{1-x^2}\right) \frac{x^2y^2}{\left( 1 -\frac{(P(x^2)-1) x^2y^2}{1-x^2}  \right)}\\
&+\left(P(x^2) - \frac{1}{1-x^2}\right)\left(\frac{P(x)-1}{1-x^2}\right)  \frac{x^3y^3}{\left( 1 -\frac{(P(x^2)-1) x^2y^2}{1-x^2}  \right)}.
\end{align*}
\end{thm}  
\begin{proof}
 First, we enumerate the nonisomorphic generalized stars on $n$ vertices. Since two g-stars are non-isomorphic if and only if the lengths of their arms differ, we notice a one-to-one correspondence between nonisomorphic generalized stars and partitions. In particular, the number of nonisomorphic g-stars on $n$ vertices is $p(n-1)$ (the $-1$ accounting for the designated central vertex), with each partition corresponding to a distinct set of possible arm lengths. Linear trees are formed from generalized stars on $\geq 2$ vertices, with intermediate paths of arbitrary length. Therefore, we will use the generating function for the number of non-isomorphic generalized stars on $\geq 2$ vertices, which is $$x\left(P(x)-1\right) = x\left(\sum_{n=0}^\infty p(n)x^n  -1\right) = \sum_{n=2}^\infty p(n-1)x^n .$$
\begin{center}
\begin{figure}[!ht]
\centering
\def\r{1.9}
\def\k{6pt}
\begin{tikzpicture}[scale=0.7,
bor/.style={circle, draw, only marks, mark=*, fill=white,mark size=\k},
rel/.style={rectangle,draw,rounded corners=0.6ex,  fill=white, minimum size=\k},
coo/.style={circle, draw, only marks, mark=*, fill=red, minimum size=\k+3pt},
]
\draw  (0,0) node[coo]{}-- (-\r,0)node[bor]{} --(0,0) --(0,\r)node[bor]{} --(0,0)--(\r,0)node[bor]{};
\end{tikzpicture}
\hspace{2cm}
\begin{tikzpicture}[scale=0.7,
bor/.style={circle, draw, only marks, mark=*, fill=white,mark size=\k},
rel/.style={rectangle,draw,rounded corners=0.6ex,  fill=white, minimum size=\k},
coo/.style={circle, draw, only marks, mark=*, fill=red, minimum size=\k+3pt},
]
\draw  (0,0) node[coo]{}-- (0,\r)node[bor]{}  --(0,2*\r)node[bor]{} --(0,0)--(\r,0)node[bor]{};
\end{tikzpicture}
\hspace{2cm}
\begin{tikzpicture}[scale=0.7,
bor/.style={circle, draw, only marks, mark=*, fill=white,mark size=\k},
rel/.style={rectangle,draw,rounded corners=0.6ex,  fill=white, minimum size=\k},
coo/.style={circle, draw, only marks, mark=*, fill=red, minimum size=\k+3pt},
]
\draw  (0,0) node[coo]{}-- (0,\r)node[bor]{}  --(0,2*\r)node[bor]{} --(0,3*\r)node[bor]{};
\end{tikzpicture}

\caption{The $p(3)=3$ non-isomorphic generalized stars on $4$ vertices (star centers in red)}
\label{fig:10-NL-tree}
\end{figure}
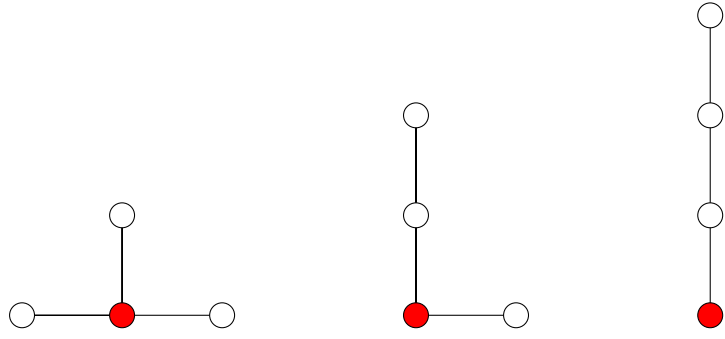
\end{center}
Let an exterior star be a generalized star at the end of the linear tree.  Such stars must have a central vertex of degree $\geq 2$, not counting the concatenating path. Therefore, there is a bijection between partitions of $n-1$ with $\geq 2$ parts and non-isomorphic exterior stars on $n$ vertices. Since there is only a single partition of $n$ with one part, $n$ itself, the generating function for exterior stars is $$x\left(P(x) - \frac{1}{1-x}\right) = x\left(\sum_{n=0}^\infty p(n)x^n  -\sum_{n=0}^\infty x^n \right) = \sum_{n=2}^\infty (p(n-1)-1)x^n.$$ Additionally, up to isomorphism, there is a unique path of length $i$, so that the generating function for the number of paths on $n$ vertices has the form $\frac{1}{1-x}$.

Therefore, the number of linear trees generated by concatenating an exterior star, $k-2$ interior stars, and a trailing exterior star, by $k-1$ paths of arbitrary length, is
\begin{align*}
\sum_{n=1}^\infty \sum_{k=0}^\infty r_{n,k}x^ny^k &= \sum_{k=2}^\infty \left(xP(x) - \frac{x}{1-x}\right)^2 \left(xP(x)-x\right)^{k-2} \frac{1}{(1-x)^{k-1}}y^k  \\
&= \left(P(x) - \frac{1}{1-x}\right)^2 \sum_{k=2}^\infty\left(P(x)-1\right)^{k-2} \frac{1}{(1-x)^{k-1}}x^ky^k \\
&= \left(P(x) - \frac{1}{1-x}\right)^2 \frac{x^2y^2}{\left( 1 -x+xy -xyP(x)  \right)}.
\end{align*} 

We now enumerate $s_{n,k}$, the number of reflectionally symmetric $k$-linear trees on $n$ vertices. These have a freely chosen central component, after which one half of the tree completely determines the other half. The component is a path when $k$ is even, and a generalized star when $k$ is odd.

If the central component is a path, it is free to have an arbitrary number of vertices, while every other component on $n$ vertices determines $2n$ vertices due to reflectional symmetry. Therefore, we count the number of $2k$-linear trees which can be generated by concatenating an exterior star, $k-1$ interior stars, a freely chosen central path, and their reflections:
\begin{align*}
\sum_{k=2}^\infty &\left(x^2P(x^2) - \frac{x^2}{1-x^2}\right)\left(x^2P(x^2)-x^2\right)^{k-2} \frac{1}{(1-x^2)^{k-2}} \frac{1}{1-x}y^{2k-2}  \\
&= \frac{1}{1-x}\left(P(x^2) - \frac{1}{1-x^2}\right) \sum_{k=2}^\infty\left(P(x^2)-1\right)^{k-2} \frac{1}{(1-x^2)^{k-2}}x^{2k-2}y^{2k-2} \\
&=\frac{1}{1-x} \left(P(x^2) - \frac{1}{1-x^2}\right) \frac{x^2y^2}{\left( 1 -\frac{(P(x^2)-1) x^2y^2}{1-x^2}  \right)}.
\end{align*}
We can conduct a similar analysis for a $(2k+1)$-linear tree, where the central component is instead a generalized star. We obtain the generating function
\begin{align*}
\sum_{k=2}^\infty &\left(x^2P(x^2) - \frac{x^2}{1-x^2}\right)\left(x^2P(x^2)-x^2\right)^{k-2} \frac{1}{(1-x^2)^{k-1}} (xP(x)-x)y^{2k-1}  \\
&= \left(P(x^2) - \frac{1}{1-x^2}\right)(P(x)-1) \sum_{k=2}^\infty\left(P(x^2)-1\right)^{k-2} \frac{1}{(1-x^2)^{k-1}}x^{2k-1}y^{2k-1} \\
&= \left(P(x^2) - \frac{1}{1-x^2}\right)\left(\frac{P(x)-1}{1-x^2}\right)  \frac{x^3y^3}{\left( 1 -\frac{(P(x^2)-1) x^2y^2}{1-x^2}  \right)}.
\end{align*}
Noting that $2a_{n,k}=r_{n,k}+s_{n,k}$ and summing all three of these generating functions completes the proof.
\end{proof}
From this generating function, we can extract the number of $k$-linear trees on $n$ vertices for small values of $n$. Table \ref{table1} displays this information for $10\leq n\leq 25$. Note that for fixed $n$, the distribution of $k$-linear trees appears to have a dominant contribution at around $k \simeq 0.2n$. As a corollary of the central limit theorem of Theorem \ref{paththm}, we will characterize this peak exactly, as lying at $\simeq 0.2192n$.

\section{Asymptotics}\label{sec2}

We wish to show that linear trees form an asymptotically small subset of all trees. Wityk \cite{Wityk} showed that the fraction of $k$-linear trees on $n$-vertices to the number of trees with $k$ high degrees vertices approaches $0$ as the number of vertices tends to infinity. However, this was only for a fixed $k$, and only partial results were shown for the natural extension to account for all linear trees. Heuristically, we expect the proportion of trees that are linear to decrease as the number of vertices increases. Given a large tree, we can color all the high degree vertices. The probability that these HDV's all lie on a single induced path intuitively decreases as the number of vertices increases. The next theorem asymptotically proves this result, and Table \ref{table2} shows this phenomenon for small values of $n$.

\begin{table}
\caption{\cite[Appendix A]{Wityk}  The percentage of nonlinear trees on $n$ vertices}\label{table2}
\begin{tabular}{r|r|r|r|r}
$n$ & Nonlinear trees & Linear Trees & \% Nonlinear & Total   \\ \hline
10  & 1               & 105          & 0.9         & 106     \\
11  & 4               & 231          & 1.7         & 235     \\
12  & 19              & 532          & 3.4         & 551     \\
13  & 47              & 1,254        & 3.6         & 1,301   \\
14  & 287             & 2,872        & 9.1         & 3,159   \\
15  & 1,002           & 6,739        & 12.9          & 7,741   \\
16  & 3,365           & 15,955       & 17.4          & 19,320  \\
17  & 10,853          & 37,776       & 22.3          & 48,629  \\
18  & 34,088          & 89,779       & 27.5          & 123,867 \\
19  & 104,574         & 213,381      & 32.9          & 317,955 \\
20  & 315,116         & 507,949      & 38.3          & 823,065 \\
21  & 935,321         & 1,209,184   & 43.6         & 2,144,505     \\
22  & 2,743,364      & 2,880,382  & 48.8         & 5,623,756 \\
23  & 7,966,723      & 6,681,351   & 53.7         & 14,828,074\\
24  & 22,951,010   & 16,348,887   & 58.4         & 39,299,897 \\
25  & 65,681,536  & 38,955,354  & 62.8        & 104,636,890 \\
\end{tabular}
\end{table} 

We can use standard techniques from analytic combinatorics to extract information about the number of nonisomorphic linear trees on $n$ vertices. In particular, we describe the asymptotic growth rate of the number of nonisomorphic linear trees, and show that the path length satisfies a central limit theorem. The methods in this section are all pulled from Flajolet and Sedgewick's monumental treatise \cite{Flajolet}. 

\begin{thm}\label{thm2}
The number of nonisomorphic linear trees on $n$ vertices, $a_n$, is asymptotically given by $$a_n \sim \frac12 \left(P(x_0) - \frac{1}{1-x_0}\right)^2 \frac{x_0}{P(x_0) + x_0P'(x_0) }  \left(  \frac{1}{x_0}\right)^n \simeq 0.7560 (2.3822)^n , $$
where $x_0 = 0.4196\ldots$ is the unique real solution, $0<x<1$, of $$x P(x) = x \prod_{i=1}^\infty \frac{1}{1-x^i} =1.$$
\end{thm}
\begin{proof}
The proof is based on meromorphic singularity analysis. We first set $y=1$ in the bivariate generating function of Theorem \ref{thm1}, giving
\begin{align*}
2\sum_{n=1}^\infty a_{n}x^n =  &\left(P(x) - \frac{1}{1-x}\right)^2 \frac{x^2}{\left( 1 -xP(x)  \right)} 
 +\frac{1}{1-x} \left(P(x^2) - \frac{1}{1-x^2}\right) \frac{x^2}{\left( 1 -\frac{(P(x^2)-1) x^2}{1-x^2}  \right)}\\
&\medspace\medspace\medspace +\left(P(x^2) - \frac{1}{1-x^2}\right)\left(\frac{P(x)-1}{1-x^2}\right)  \frac{x^3}{\left( 1 -\frac{(P(x^2)-1) x^2}{1-x^2}  \right)} \\
&= \left(P(x) - \frac{1}{1-x}\right)^2 \frac{x^2}{1 -xP(x) } 
 +\left(P(x^2) - \frac{1}{1-x^2}\right) \frac{x^2 (1+x)}{ 1 - x^2P(x^2)  }\\
&\medspace\medspace\medspace +\left(P(x^2) - \frac{1}{1-x^2}\right)\left(\frac{P(x)-1}{1-x^2}\right)  \frac{x^3 (1-x^2)}{ 1 - x^2P(x^2)  }.
\end{align*}
We know that $P(x)$ is analytic for $x \in \mathbb{C}, |x|<1$. Therefore, the only poles inside the unit disc arise from the denominator terms of $1-xP(x)$ and $1-x^2P(x^2)$. Since $xP(x)$ is strictly increasing on the real line and $ \lim_{x\to 1^{-}} xP(x) = \infty$, there is a unique real root $x_0 = 0.4196\ldots$ to the equation $xP(x)=1$ satisfying $0< x_0 <1$. The denominator term of $1-x^2P(x^2)$ contributes a singularity at $x_0^{\frac12} = 0.6478\ldots > x_0$. Also, the pole at $x_0$ is the only pole on the circle $x = |x_0|$, since $|x P(x)| < x_0 P(x_0)$ if $x\neq x_0$. Therefore, $x_0$ is the dominant singularity, in that it is the pole with smallest absolute value. 

Therefore, appealing to the methods of \cite[Chapter IV]{Flajolet}, we immediately have that $$2a_n \sim   \left(P(x_0) - \frac{1}{1-x_0}\right)^2 \frac{x_0}{P(x_0) + x_0P'(x_0) }  \left(  \frac{1}{x_0}\right)^n.$$
Inn particular, note that in the language of \cite[Chapter V.2, p. 294]{Flajolet}, we are dealing with a \textit{supercritical sequence} with $G(x) = xP(x)$ and $G'(x) = P(x) + xP'(x)$. Hence the result follows directly from \cite[Theorem V.1]{Flajolet}.
\end{proof}

We can then obtain statistics about the number of HDV's in a random linear tree, by conducting another singularity analysis of the generating function of Theorem \ref{thm1}. We will apply the moving pole analysis of Flajolet and Sedgewick. In what follows, for any function $f(u)$ analytic at $u=1$ and satisfying $f(1) \neq 0$, we set
\begin{equation}\label{meanvar}
\mathfrak{m}(f) = \frac{f'(1)}{f(1)}, \hspace{1cm} \mathfrak{v}(f)= \frac{f''(1)}{f(1)} +  \frac{f'(1)}{f(1)} - \left(\frac{f'(1)}{f(1)}\right)^2. 
\end{equation}
We will appeal to the following theorem to prove our main result.
\begin{thm}\cite[Thm IX.12 (Algebraic singularity schema)]{Flajolet}\label{thm9.12}
Let $F(x,y)$ be a function that is bivariate analytic at $(x,y) = (0,0)$ and has non-negative coefficients. Assume also the following conditions: 
\begin{enumerate} 
\item Analytic perturbation: there exist three functions $A, B, C$ analytic in a domain $\mathcal{D} = \{ |x| \leq r \}  \times \{|y-1| < \epsilon\}$, such that, for some $r_0$ with $0<r_0 \leq r$ and $\epsilon >0$, the following representation holds, with $\alpha \not\in \mathbb{Z}_{\leq 0}$,
$$F(x,y) = A(x,y) + B(x,y)C(x,y)^{-\alpha}; $$
furthermore, assume that in $|x| \leq r$, there exists a unique root $\rho$ of the equation $C(x,1)=0$, that this root is simple, and that $B(\rho,1) \neq 0$.
\item Non-degeneracy: one has $\partial_xC(\rho,1) \cdot \partial_yC(\rho,1) \neq 0$, ensuring the existence of a non-constant $\rho(y)$ analytic at $y=1$, such that $C(\rho(y),y)=0$ and $\rho(1) = \rho$.
\item Variability: one has $$\mathfrak{v} \left(  \frac{\rho(1)}{\rho(y)}\right)\neq 0. $$
\end{enumerate}
Then the random variable with probability generating function $\frac{[x^n]F(x,y)}{[x^n]F(x,1)}$ converges in distribution to a Gaussian variable with a speed of convergence that is $O(n^{-\frac12})$. The mean and variance \textnormal{[corrected]} are asymptotically linear in $n$.
\end{thm}

\begin{thm}\label{paththm}
Define the mean and variance
\begin{align*}
 \mu &=  \frac{P(x_0)-1}{ P(x_0)+x_0P'(x_0)  } \simeq 0.219273\ldots\\
\sigma^2 &= \frac{(x_0-1)    \left(  -1-P'(x_0)x_0+P'(x_0)x_0^2+P'(x_0)^2x_0^3-P''(x_0)x_0^2+P''(x_0)x_0^3  \right) }{   x_0^2(P(x_0) +x_0P'(x_0) )^3 } \simeq 0.0567065\ldots.
\end{align*}
For large $n$, $\frac{a_{n,k}}{a_n}$ converges in distribution to a Gaussian distribution with mean $\mu n$ and variance $\sigma^2 n$, with speed of convergence $O(n^{-\frac12})$, i.e. the normalized random variable
$$ \frac{1}{\sigma\sqrt{n}} \left(  \frac{a_{n,k}}{a_n}- \mu n \right) $$
converges in distribution to a standard normal distribution.
\end{thm}
\begin{proof}

We again refer to Theorem \ref{thm1}, that
\begin{align}
2\sum_{n=1}^\infty \sum_{k=0}^\infty a_{n,k}x^ny^k &=  \left(P(x) - \frac{1}{1-x}\right)^2 \frac{x^2y^2}{\left( 1 -x+xy -xyP(x)  \right)} \label{deriv1}\\
&+\frac{1}{1-x} \left(P(x^2) - \frac{1}{1-x^2}\right) \frac{x^2y^2}{\left( 1 -\frac{(P(x^2)-1) x^2y^2}{1-x^2}  \right)} \nonumber \\
&+\left(P(x^2) - \frac{1}{1-x^2}\right)\left(\frac{P(x)-1}{1-x^2}\right)  \frac{x^3y^3}{\left( 1 -\frac{(P(x^2)-1) x^2y^2}{1-x^2}  \right)}. \nonumber
\end{align}
At $y=1$, Equation \eqref{deriv1} reduces to
\begin{align*}
&  \left(P(x) - \frac{1}{1-x}\right)^2 \frac{x^2}{\left( 1 -xP(x)  \right)^2} \\
&+\frac{1}{1-x} \left(P(x^2) - \frac{1}{1-x^2}\right) \frac{ x^2(1-x^2) }{\left( 1 -x^2 P(x^2) \right)}\\
&+\left(P(x^2) - \frac{1}{1-x^2}\right)\left(\frac{P(x)-1}{1-x^2}\right) \frac{x^3(1-x^2)}{\left( 1 -x^2 P(x^2)\right)}.
\end{align*}
By the same singularity analysis as that of Theorem \ref{thm2}, we see that the dominant pole occurs at $x_0 = 0.4196\ldots$ again. Thus at $y=1$ we take $\rho = x_0 = 0.4196\ldots$ as before. By inspecting Equation \eqref{deriv1}, we see that for $y$ sufficiently close to $1$, the dominant singularity will arise from the term with denominator $(1-x+xy-xyP(x))$, and the other two terms will be analytic in a sufficiently small neighborhood of $(x_0,1)$. We can thus appeal to Theorem \ref{thm9.12}, where we take 
\begin{align*} 
A(x,y) &= \frac12\frac{1}{1-x} \left(P(x^2) - \frac{1}{1-x^2}\right) \frac{x^2y^2}{\left( 1 -\frac{(P(x^2)-1) x^2y^2}{1-x^2}  \right)} \\
 &\hspace{4mm}+\frac12\left(P(x^2) - \frac{1}{1-x^2}\right)\left(\frac{P(x)-1}{1-x^2}\right)  \frac{x^3y^3}{\left( 1 -\frac{(P(x^2)-1) x^2y^2}{1-x^2}  \right)}, \\
B(x,y) &= \frac12\left(P(x) - \frac{1}{1-x}\right)^2 {x^2y^2(2-2x+xy-xyP(x))}, \\
C(x,y) &=  1 -x+xy -xyP(x).
\end{align*}


Setting 
$$c_{i,j} :=   \left.\frac{\partial^{i+j}}{\partial x^i\partial y^j} C(x,y) \right\vert_{(\rho,1)},$$ 
we find 
\begin{align*}
c_{1,0} &= -P(x_0) - x_0 P'(x_0)   = -6.3082\ldots ,\\
c_{0,1} &= x_0 - x_0P(x_0) = x_0 -1  = -0.5804\ldots,\\
c_{1,1} &= 1 - P(x_0) -x_0P'(x_0)  = -5.3082\ldots,\\
c_{2,0} &=  -2P'(x_0) -x_0P''(x_0) = -49.5223\ldots,\\
c_{0,2} &= 0.\\
\end{align*}
The non-degeneracy condition then simplifies to $c_{0,1} c_{1,0} = 3.6612\ldots \neq 0$. Furthermore, we can solve for the local expansion of the functional equation $C(\rho(y),y) = 0$ around $y=1$, by using standard series reversion techniques \cite[Equation (38), p. 672]{Flajolet} to find
\begin{align*}
\rho(y)  = \rho - \frac{c_{0,1}}{c_{1,0}} (y-1) - \frac{  c_{1,0}^2 c_{0,2} - 2 c_{1,0} c_{1,1}c_{0,1}+c_{2,0}c_{0,1}^2}{  2 c_{1,0}^3} (y-1)^2 + O((y-1)^3)
\end{align*}
and thus 
\begin{align*}
\frac{\rho(y)}{\rho}  &= 1 -  \frac{P(x_0)-1}{ P(x_0)+x_0P'(x_0)  }(y-1)  \\
&-   \frac{(x_0-1) \left(  2P(x_0)^2+2P'(x_0) -P''(x_0) x_0 +P''(x_0)x_0^2 +2P'(x_0)^2x_0^2 -2P(x_0)   \right) }{  2x_0(P(x_0) +P'(x_0)x_0 )^3 }   (y-1)^2 \\
&+ O((y-1)^3) \\
& := 1 - \gamma (y-1) - \delta (y-1)^2 +O((y-1)^3).
\end{align*}
Then, we can expand the inverse to second order, which gives
\begin{equation*}
\frac{\rho}{\rho(y)}  = 1 + \gamma (y-1) + (\gamma^2+\delta) (y-1)^2 +O((y-1)^3)
\end{equation*}

Referring back to definition \eqref{meanvar}, we further deduce
$$\mathfrak{m}\left(   \frac{\rho}{\rho(y)} \right)  = \gamma, \hspace{4mm} \mathfrak{v}\left(   \frac{\rho}{\rho(y)} \right)  = 2(\delta+\gamma^2) +\gamma-\gamma^2 = 2\delta + \gamma + \gamma^2,$$
which expand to the values of $\mu$ and $\sigma^2$ given in the statement of the theorem. Numerically, we have $\mathfrak{\mu} = 0.2192\ldots$ as expected from the numerical data, and $\mathfrak{v} = 0.0567\ldots \neq 0$, and the variance condition is also verified. Finally, we appeal to a general remark \cite[p. 678]{Flajolet} that the asymptotic mean and the variance of our distribution are given exactly by Equation \eqref{meanvar}.
\end{proof}

\section*{Acknowledgements}
Part of this work was carried out by the second author as part of his honors thesis at the College of William and Mary. T.W. would also like to thank Roberto Costas-Santos for being an excellent advisor during the College of William and Mary Matrix REU, where part of this work was completed. We would also like to thank the anonymous referees, who suggested the proof of the central limit theorem, and Larry Washington, who finally found a very pernicious error in our proof of the central limit theorem.

\end{document}